\documentclass[11pt,a4paper]{amsart}

\usepackage{amsmath,amsfonts,amssymb,amsthm,geometry,graphicx,tikz,pgfplots,xcolor,tikz-cd,pdfpages,upgreek,hyphenat,csquotes,oubraces,accents}
\usepackage[utf8]{inputenc}
\usepackage[square,numbers,sort&compress]{natbib}
\usepackage[hidelinks]{hyperref}
\usepackage[UKenglish]{babel}
\graphicspath{{images/}{./}}

\usepackage{newtxtext,newtxmath,xpatch}
\patchcmd{\part}{\sfdefault}{}{}
\patchcmd{\specialsection}{\sfdefault}{}{}
\patchcmd{\section}{\sfdefault\scshape}{}{}
\patchcmd{\subsection}{\sfdefault}{}{}
\patchcmd{\subsubsection}{\sfdefault}{}{}

\theoremstyle{plain}
\newtheorem{theorem}{Theorem}[section]
\newtheorem{lemma}[theorem]{Lemma}
\newtheorem{cor}[theorem]{Corollary}
\newtheorem{prop}[theorem]{Proposition}
\newtheorem{qtn}[theorem]{Question}
\newtheorem*{metaconj}{Metaconjecture}

\theoremstyle{definition}
\newtheorem{defn}[theorem]{Definition}

\newtheorem{remark}[theorem]{Remark}

\DeclareMathOperator{\Aut}{\mathrm{Aut}}
\DeclareMathOperator{\Mod}{\mathrm{Mod}}
\DeclareMathOperator{\bsd}{\mathrm{bsd}}
\DeclareMathOperator{\Stab}{\mathrm{Stab}}
\DeclareMathOperator{\Image}{\mathrm{Im}}
\DeclareMathOperator{\Sym}{\mathrm{Sym}}
\DeclareMathOperator{\Comm}{\mathrm{Comm}}
\DeclareMathOperator{\hdist}{\mathrm{hdist}}

\begin{document}

\title{The curve complex as a coset intersection complex}
\author[]{Haoyang He and Eduardo Martínez-Pedroza}  
\address{Department of Mathematics and Statistics, Memorial University of Newfoundland, St. John's, NL, A1C 5S7, Canada}
\email{haoyangh@mun.ca, emartinezped@mun.ca}
\date{\today}

\begin{abstract}
	We show that there is a collection of subgroups of the mapping class group of a surface such that the associated coset intersection complex is quasi-isometric and homotopy equivalent to the curve complex. Moreover, we prove that these two complexes are combinatorially equivalent in the sense that one can be obtained from the other via taking a nerve. As an application, we prove that the automorphism group of this coset intersection complex is the extended mapping class group, a result in the spirit of for Ivanov's metaconjecture. 
\end{abstract}

\maketitle

\section{Introduction}

A \emph{group pair} $(G,\mathcal{A})$ consists of a finitely generated group $G$ and a finite collection $\mathcal{A}$ of infinite subgroups. We write $G/\mathcal{A}=\{gA\mid g\in G,A\in\mathcal{A}\}$. The following simplicial complex associated to a group pair $(G,\mathcal{A})$ was introduced by Abbott\textendash Martínez-Pedroza~\cite{MR5034311}. 

\begin{defn}[{\cite[Definition 4.1]{MR5034311}}]
	\label{def_cic}
	The \emph{coset intersection complex} of $(G,\mathcal{A})$, denoted by $\mathcal{K}(G,\mathcal{A})$, is the simplicial complex whose vertex set is $G/\mathcal{A}$, and $\{g_0A_0,\cdots,g_kA_k\}\subseteq G/\mathcal{A}$ is a simplex if $\bigcap_{i=0}^kg_iA_ig_i^{-1}$ is infinite.
\end{defn}

It is worth noting that a similar simplicial complex, the coset complex, was introduced by Abels\textendash Holz~\cite{MR1244917} and was studied, for instance, by Brück~\cite{MR4411324} in the context of the outer automorphism group of a right-angled Artin group. 

This article concerns certain coset intersection complexes of the mapping class group of a surface. We are motivated by the following two results. Together with results in this article, they suggest that it may be fruitful to study the coset intersection complex for other class of groups and their associated complexes that are ubiquitous in geometric group theory. 

\begin{theorem}[{\cite[Theorem 6.1]{MR5034311}}]
	\label{thm_AMP256.1}
	Let $\Gamma$ be a finite triangle-free simplicial graph such that every vertex has valence at least two. Let $\mathcal{B}$ be the collection of maximal abelian subgroups of the right-angled Artin group $A(\Gamma)$. Then $\mathcal{K}(A(\Gamma),\mathcal{B})$ is quasi-isometric to the extension graph $E(\Gamma)$.
\end{theorem}

The extension graph was used by Kim\textendash Koberda \cite{MR3039768} to study the embedability between right-angled Artin groups and it is quasi-isometric to a tree \cite[Lemma 3.5(7)]{MR3039768}. 

\begin{theorem}[{\cite[Proposition 4.8]{abbott2025homotopytypescomplexeshyperplanes}}]
	\label{thm_AGMP254.8}
	Let $\Gamma$ be a finite connected simplicial graph. Let $\mathcal{G}$ be a collection of infinite subgroups indexed by vertices of $\Gamma$. Let $\mathcal{B}$ be the collection of subgroups of the graph product $\Gamma\mathcal{G}$ generated by cliques of $\Gamma$. Then $\mathcal{K}(\Gamma\mathcal{G},\mathcal{B})$ is quasi-isometric and homotopy equivalent to the crossing complex $\mathrm{Cross}^{\Delta}(\Gamma,\mathcal{G})$.  
\end{theorem}

We refer to~\cite[Section 3]{abbott2025homotopytypescomplexeshyperplanes} for the definition of the crossing complex. The one-skeleton of the crossing complex, the crossing graph, was used by Genevois to obtain useful results in relation to the graph product \cite[Section 8]{genevois2017cubicallikegeometryquasimediangraphs}. Abbott\textendash Genevois\textendash Martínez-Pedroza used the crossing complex to establish a necessary condition for the quasi-isometric rigidity of right-angled Artin groups \cite[Theorem 1.1]{abbott2025homotopytypescomplexeshyperplanes}.

We now recall definitions of the mapping class group and the curve complex in order to state our results. Throughout this article, $S$ is a closed, connected, oriented surface of finite type with genus at least two. The \emph{mapping class group} of $S$, denoted by $\Mod(S)$, is the group of isotopy classes of orientation-preserving homeomorphisms $S\to S$. The \emph{extended mapping class group} of $S$, denoted by $\Mod^{\pm}(S)$, is the group of isotopy classes of all homeomorphisms $S\to S$. The \emph{curve complex} $\mathcal{C}(S)$ of $S$, as defined by Harvey \cite[page 246]{MR624817}, is the simplicial complex whose vertices are isotopy classes of essential simple closed curves in $S$ and whose simplices are sets of isotopy classes that can be realised disjointly. The curve complex is a powerful tool in low-dimensional topology. For instance, it is a key ingredient in establishing the quasi-isometric rigidity of $\Mod^{\pm}(S)$ \cite{MR2928983}, in finding the curvature and rank of the Teichmüller space \cite{MR2197066}, and in proving the ending lamination conjecture \cite{MR2630036,MR2925381}.

One of the results in this article, in a similar fashion to Theorems~\ref{thm_AMP256.1} and~\ref{thm_AGMP254.8}, is as follows. Recall that the action of $\Mod(S)$ on the set $\mathcal{P}=\{P\mid P\text{ is a maximal simplex in }\mathcal{C}(S)\}$ has finitely many orbits, say $n$. Let $P_1,\cdots,P_n$ be representatives of the orbits of this action. Denote $T_x$ as the \emph{Dehn twist} about the isotopy class $x$ of simple closed curves in $S$. Let $H_i=\langle T_x\mid x\in P_i\rangle$ for each $1\le i\le n$. Note that each $H_i$ is a free abelian group of $\Mod(S)$ of maximal rank \cite[Theorem 1]{MR726319}. Let $\mathcal{H}=\{H_i\mid 1\le i\le n\}$. Let $\mathcal{T}=\{\Stab(P_i)\mid1\le i\le n\}$, where $\Stab(P_i)$ is the subgroup of $\Mod(S)$ that leaves the simplex $P_i$ invariant. 

\begin{theorem}
	\label{thm_qihe}
	The curve complex $\mathcal{C}(S)$ is quasi-isometric and homotopy equivalent to both $\mathcal{K}(\Mod(S),\mathcal{T})$ and $\mathcal{K}(\Mod(S),\mathcal{H})$. 
\end{theorem}

We also establish a combinatorial equivalence between $\mathcal{C}(S)$ and $\mathcal{K}(\Mod(S),\mathcal{T})$. Recall that, given a collection $\mathcal{U}$ of non-empty subsets of a set, the \emph{nerve} of $\mathcal{U}$ is the simplicial complex whose vertex set is $\mathcal{U}$ and such that $\{U_0,\cdots,U_k\}\subseteq\mathcal{U}$ is a simplex if $\bigcap_{i=0}^kU_i$ is non-empty. For a simplicial complex $X$, denote $\mathcal{N}(X)$ as the nerve of the collection of maximal simplices of $X$. The complexes $\mathcal{C}(S)$ and $\mathcal{K}(\Mod(S),\mathcal{T})$ are combinatorially equivalent in the following sense.

\begin{theorem}
	\label{thm_combequi}
	\begin{enumerate}
		\item[]
		\item $\mathcal{N}(\mathcal{C}(S))$ is isomorphic to $\mathcal{K}(\Mod(S),\mathcal{T})$.
		\item $\mathcal{N}(\mathcal{K}(\Mod(S),\mathcal{T}))$ is isomorphic to $\mathcal{C}(S)$. 
	\end{enumerate}
\end{theorem}

The following corollary is an analogous statement for groups quasi-isometric to $\Mod(S)$, following quasi-isometric rigidity of $\Mod^{\pm}(S)$ \cite{MR2928983,hamenstaedt2007geometrymappingclassgroups} and results in \cite{MR5034311}.

\begin{cor}
	\label{thm_gqimcg}
	Let $H$ be a finitely generated group quasi-isometric to $\Mod(S)$. Then there exists a finite collection $\mathcal{R}$ of subgroups of $H$ such that the following hold.
	\begin{enumerate}
		\item The curve complex $\mathcal{C}(S)$ is isomorphic to  $\mathcal{N}(\mathcal{K}(H,\mathcal{R}))$. 
		\item The coset intersection complex $\mathcal{K}(H,\mathcal{R})$ is isomorphic to  $\mathcal{N}(\mathcal{C}(S))$.
	\end{enumerate}
\end{cor}

We are also interested in the automorphism group of $\mathcal{K}(\Mod(S),\mathcal{T})$. The motivation is from the following important result, which is due, independently, to Ivanov \cite{MR1460387} and Luo \cite{MR1722024} (cases of surfaces with genera zero and one were addressed in \cite{MR1696431,MR1722024}). 

\begin{theorem}[Ivanov, Luo]
	\label{thm_autcs}
	Let $S$ be a surface with genus at least three. Then $\Aut(\mathcal{C}(S))$ is isomorphic to $\Mod^{\pm}(S)$. 
\end{theorem}

If $S$ is the genus two surface, this theorem does not hold since the hyperelliptic involution fixes every vertex of $\mathcal{C}(S)$. There are a number of complexes associated to a surface (possibly with boundary or punctures) that has automorphism group $\Mod^{\pm}(S)$. These include the pants complex \cite{MR2040283}, the complex of separating curves \cite{MR2119299,MR2793105}, the complex of non-separating curves \cite{MR2214789}, the Schmutz graph of nonseparating curves \cite{MR1781329}, the arc complex \cite{MR2681579}, the arc and curve complex \cite{MR2609303}, the Torelli complex \cite{MR2793105}, the Hatcher\textendash Thurston complex \cite{MR2365859}, the ideal triangulation graph \cite{MR3025746}, the truncated complex of domains \cite{MR2952767}, and the polygonalisation complex \cite{MR3968817}. The complex of regions studied by Brendle\textendash Margalit \cite{MR4013739} covers many of aforementioned complexes, and the work of Disarlo\textendash Koberda\textendash de la Nuez González \cite{disarlo2023modeltheorycurvegraph} interprets the metaconjecture from a model-theoretic perspective. Many proofs of these results appeal to Theorem~\ref{thm_autcs}. In response to (some of) these results, Ivanov proposed the following metaconjecture \cite[Section 6]{MR2264532}.

\begin{metaconj}
	Every object that is naturally associated to a surface $S$ and with a sufficiently rich structure has automorphism group isomorphic to $\Mod^{\pm}(S)$. Moreover, this can be proved by a reduction to Theorem~\ref{thm_autcs}. 
\end{metaconj}

Motivated by this metaconjecture, we prove the following. 

\begin{theorem}
	\label{thm_autksts}
	Let $S$ be a surface with genus at least three. Then $\Aut(\mathcal{K}(\Mod(S),\mathcal{T}))$ is isomorphic to $\Mod^{\pm}(S)$. Consequently, for every finitely generated group $H$ quasi-isometric to $\Mod(S)$, there is a finite collection $\mathcal{Q}$ of subgroups of $H$ such that $\Aut(\mathcal{K}(H,\mathcal{Q}))$ is isomorphic to $\Mod^{\pm}(S)$.
\end{theorem}

Theorem~\ref{thm_autksts} follows from Theorems~\ref{thm_combequi}(1),~\ref{thm_autcs}, Corollary~\ref{thm_gqimcg}, and the following theorem. Note that there are flag simplicial complexes $X$ such that the natural map $\Aut(X)\to \Aut(\mathcal{N}(X))$ is not an isomorphism, but this is not the case for the curve complex. 

\begin{theorem}
	\label{thm_metacicmodsT}
	$\Aut(\mathcal{C}(S))$ is isomorphic to $\Aut(\mathcal{N}(\mathcal{C}(S)))$.
\end{theorem}

In fact, we will show a more general result, see Theorem~\ref{thm_assumpautxiso}. In view of Ivanov's metaconjecture and Theorem~\ref{thm_autksts}, the following question naturally arises.

\begin{qtn}
	Is $\Aut(\mathcal{K}(\Mod(S),\mathcal{H}))$ isomorphic to $\Mod^{\pm}(S)$? More generally, for a finite collection $\mathcal{A}$ of infinite index subgroups of $\Mod(S)$, when is $\Aut(\mathcal{K}(\Mod(S),\mathcal{A}))$ isomorphic to $\Mod^{\pm}(S)$? 
\end{qtn}

\subsection*{Organisation} 

We prove Theorem~\ref{thm_combequi}(1) in Section~\ref{sec_ncskst}. More general versions of Theorems~\ref{thm_combequi}(2) and ~\ref{thm_metacicmodsT} are established in Section~\ref{sec_autx}. Corollary~\ref{thm_gqimcg} is proved in Section~\ref{sec_gqimcgs}, and Theorem~\ref{thm_qihe} is shown in Section~\ref{sec_qihe}. We conclude by constructing a simplicial embedding from $\mathcal{C}(S)$ to another coset intersection complex of $\Mod(S)$ in Section~\ref{sec_ccsubcic}. 

\subsection*{Acknowledgements} 

We would like to thank Saul Schleimer for his suggestion that streamlines the proof of Corollary~\ref{thm_gqimcg}, Benjamin Brück for bringing the coset complex to our attention, and Robert Tang for pointing out a reference. We would also like to thank Russ Woodroofe for comments. HH is partially supported by the School of Graduate Studies at Memorial University of Newfoundland. EMP acknowledges funding by the Natural Sciences and Engineering Research Council (NSERC) of Canada.

\section{The nerve of the curve complex}
\label{sec_ncskst}

In this section, we prove Theorem~\ref{thm_combequi}(1). Recall the following definition. 

\begin{defn}
	An \emph{(abstract) simplicial complex} $X$ is a set whose elements are non-empty finite subsets, and such that if $\emptyset\neq\tau\subseteq\sigma\in X$ then $\tau\in X$. Elements of $X$ are called \emph{simplices} of $X$, and elements of $\bigcup X$ are called \emph{vertices} of $X$. Given two simplices $\tau$ and $\sigma$ with $\tau\subseteq\sigma$, we say that $\tau$ is a \emph{face} of $\sigma$. 
	
	The \emph{dimension} of a simplex $\sigma$ is its cardinality, and the dimension of $X$ is the maximum of the dimension of its simplices, which could be infinite. A \emph{maximal simplex $\Delta$} of $X$ is a maximal collection of vertices with the property that any finite subset of $\Delta$ is a simplex of $X$. Note that if $X$ is infinite dimensional, then a maximal simplex may not be a simplex of $X$. The \emph{$r$-skeleton} of $X$, denoted by $X^{[r]}$, is the collection of simplices in $X$ with dimension at most $r$. 
	
	A simplicial map $f\colon X\to Y$ between simplicial complexes is a function $f$ from the vertex set of $X$ to the vertex set of $Y$ such that the image of each simplex of $X$ is a simplex of $Y$. 
\end{defn}

We summarise some standard results about Dehn twists: see \cite[Chapter 3]{MR2850125} for detailed discussions. Recall that, given $a,b\in\mathcal{C}^{[0]}(S)$, the \emph{geometric intersection number} of $a$ and $b$ is $i(a,b)=\min\{\vert\alpha\cap\beta\vert\mid\alpha\in a,\beta\in b\}$.

\begin{lemma}
	\label{lem_Dtp}
	Let $a,b$ be vertices of $\mathcal{C}(S)$. Then we have the following:
	\begin{enumerate}
		\item[(1)] If $a$ is an isotopy class of essential simple closed curves in $S$, then $T_a$ is of infinite order. 
		\item[(2)] For every $s,t\in\mathbb{Z}-\{0\}$, $T^s_a=T^t_b$ if and only if $a=b$ and $s=t$.
		\item[(3)] $fT_af^{-1}=T_{f(a)}$ for every $f\in\Mod(S)$.  
		\item[(4)] $i(a,b)=0$ if and only if $T_aT_b=T_bT_a$ if and only if $T_a(b)=b$.  
		\item[(5)] Given $k\in\mathbb{Z}$, we have $i(T_a^k(b),b)=\vert k\vert(i(a,b))^2$. 
	\end{enumerate}
\end{lemma}

The following lemma is a generalisation of Lemma~\ref{lem_Dtp}(2).

\begin{lemma}
	\label{lem_3.17FM12}
	Let $m,k\in\mathbb{N}$. Let $\{a_1,\cdots,a_m\}$ and $\{b_1,\cdots,b_k\}$ be simplices in $\mathcal{C}(S)$. Let $p_i,q_j\in\mathbb{Z}-\{0\}$ where $1\le i\le m$ and $1\le j\le k$. Then if \[T_{a_1}^{p_1}T_{a_2}^{p_2}\cdots T_{a_m}^{p_m}=T_{b_1}^{q_1}T_{b_2}^{q_2}\cdots T_{b_k}^{q_k}\] in $\Mod(S)$, then $m=k$ and $\{T_{a_i}^{p_i}\mid 1\le i\le m\}=\{T_{b_i}^{q_i}\mid 1\le i\le k\}$. 
\end{lemma}

It follows from Lemmata~\ref{lem_Dtp}(1)(4) and \ref{lem_3.17FM12} that, for each $1\le i\le n$, $H_i\cong\mathbb{Z}^{3g-3}$, where $g\ge2$ is the genus of surface $S$. By definition of Dehn twists, $H_i$ is a subgroup of $\Stab(P_i)$. 

Given a collection of simple closed curves $C$, the \emph{cut surface} $S_C$ is obtained from $S$ by removing a regular neighbourhood of $C$. A \emph{pants decomposition} of $S$ is a collection of disjoint simple closed curves $P$ in $S$ such that $S_P$ consists of disjoint copies of pairs of pants, which are spheres with three holes. An Euler characteristic argument shows that a pants decomposition $P$ is a collection of $3g-3$ disjoint, non-isotopic essential simple closed curves and $S_P$ is a disjoint union of $2g-2$ pairs of pants. Note that, up to isotopy of each curve, the collection of pants decompositions is in bijective correspondence with the collection maximal simplices of $\mathcal{C}(S)$.

\begin{lemma}
	\label{lem_Hfistab}
	Let $P$ be a maximal simplex in $\mathcal{C}(S)$. Let $H=\langle T_a\mid a\in P\rangle$. Then $H$ is a finite-index subgroup of $\Stab(P)$.
\end{lemma}

\begin{proof}
	Define $\Stab^{\mathrm{pw}}(P)=\{f\in\Mod(S)\mid f(a)=a,\forall a\in P\}$. We will show that $H$ is a finite-index subgroup of $\Stab^{\mathrm{pw}}(P)$, which is finite index in $\Stab(P)$. Identify $P$ with a pants decomposition of $S$. 
	
	Let $\varphi\in\Stab(P)$. This induces a permutation $\varphi_*\colon P\to P$. Define $f\colon\Stab(P)\to\Sym(P)$ by $\varphi\mapsto\varphi_*$. Then $\ker(f)=\Stab^{\mathrm{pw}}(P)$, so we have a short exact sequence:
	\begin{center}
		\begin{tikzcd}
			1\ar[r] & \Stab^{\mathrm{pw}}(P)\ar[r,hookrightarrow] & \Stab(P)\ar[r,"f"] & \Image(f)\ar[r] & 1.
		\end{tikzcd}
	\end{center}
	Since $\Image(f)$ is finite, $\Stab^{\mathrm{pw}}(P)$ is a finite-index subgroup of $\Stab(P)$.
	
	We now show that $H$ is a finite index subgroup of $\Stab^{\mathrm{pw}}(P)$ by proving that $H$ is the kernel of a group homomorphism $h\colon\Stab^{\mathrm{pw}}(P)\to\prod_{i=1}^{3g-3}(\mathbb{Z}/2\mathbb{Z})$. To define $h$, let $\gamma\in\Stab^{\mathrm{pw}}(P)$. Consider a collection of representatives $\{a_1,\cdots,a_{3g-3}\}$ of the isotopy classes of the curves in $P$ and consider the corresponding cut surface which for simplicity we denote by $S_P$. Consider a representative of the mapping class $\gamma$ that fixes setwise a regular neighbourhood of each of the curves $\{a_1,\cdots,a_{3g-3}\}$, for simplicity we denote this homeomorphism by $\gamma$. Denote two boundary components of such regular neighbourhood of $a_i$ in $S$ as $a_i^+$ and $a_i^-$. Then either $\gamma_*(a_i^+) = a_i^+$ and  $\gamma_*(a_i^-)=a_i^-$, or $\gamma_*(a_i^+)=a_i^-$ and $\gamma_*(a_i^-)=a_i^+$ in $S_P$. Define $h\colon\Stab^{\mathrm{pw}}(P)\to\prod_{i=1}^{3g-3}(\mathbb{Z}/2\mathbb{Z})$ by $h(\gamma)=(h_1,\cdots,h_{3g-3})$, where $h_i=0$ if $\gamma_*(a_i^+)=a_i^+$ and  $\gamma_*(a_i^-)=a_i^-$, and $h_i=1$ if $\gamma_*(a_i^+)=a_i^-$ and $\gamma_*(a_i^-)=a_i^+$. Note that $H\subseteq\ker(h)$. Let $\sigma\in\ker(h)$ and consider a representative that fixes setwise each of the curves $a_i$. Let $S_1,\cdots,S_{2g-2}$ be pairs of pants in $S_P$. Then $\sigma_*(S_j)=S_j$ for every $1\le j\le2g-2$, and $\sigma_*(a_i^+)=a_i^+$ and $\sigma_*(a_i^-)=a_i^-$. This means $\sigma_*$ fixes every boundary component of $S_j$, so $\sigma_*\arrowvert_{S_j}$ is a composition of Dehn twists about boundary components of $S_j$. This forces $\sigma\in H$, since $T_{\alpha_i^+}=T_{\alpha_i^-}=T_{\alpha_i}$ in $\Mod(S)$, so $\ker(h)\subseteq H$. Therefore, we have a short exact sequence
	\begin{center}
		\begin{tikzcd}
			1\ar[r] & H\ar[r,hookrightarrow] & \Stab^{\mathrm{pw}}(P)\ar[r,"h"] & \Image(h)\ar[r] & 1. 
		\end{tikzcd}
	\end{center}
	Since $\Image(h)$ is finite, $H$ is a finite-index subgroup of $\Stab^{\mathrm{pw}}(P)$.
\end{proof}

Recall that $\mathcal{P}$ denotes the collection of all maximal simplices of $\mathcal{C}(S)$.

\begin{lemma} 	
	\label{lem_comgen}
	Let $Q_0,\cdots,Q_k\in\mathcal{P}$. For each $0\le i\le k$, let $G_i=\langle T_a\mid a\in Q_i\rangle$. Let $g_0,\cdots,g_k\in\Mod(S)$. Then the following are equivalent:
	\begin{enumerate}
		\item $\bigcap_{i=0}^kg_i\Stab(Q_i)g_i^{-1}$ is infinite.
		\item $\bigcap_{i=0}^kg_iG_ig_i^{-1}$ is infinite. 
		\item $\{g_0Q_0,\cdots,g_kQ_k\}$ is a simplex of $\mathcal{N}(\mathcal{C}(S))$, that is, $\bigcap_{i=0}^kg_i(Q_i)$ is non-empty.
	\end{enumerate}
	Moreover, $\bigcap_{i=0}^kg_iG_ig_i^{-1}=\langle T_a\mid a\in\bigcap_{i=0}^kg_i(Q_i)\rangle$.
\end{lemma}

\begin{proof}
	Note that $g_iG_ig_i^{-1}=\langle T_a\mid a\in g_i(Q_i)\rangle$ for each $0\le i\le k$,   
	see Lemma~\ref{lem_Dtp}(3).
	
	We first prove that (2) implies (1). Note that if $\bigcap_{i=0}^kg_iG_ig_i^{-1}$ is infinite then $\bigcap_{i=0}^kg_i\Stab(Q_i)g_i^{-1}$ is infinite, since each $G_i$ is a subgroup of $\Stab(Q_i)$, and therefore 
	$\bigcap_{i=0}^kg_i\Stab(Q_i)g_i^{-1}$ is infinite. That (1) implies (2) is a consequence of $\bigcap_{i=0}^kg_iG_ig_i^{-1}$ being a finite-index subgroup of $\bigcap_{i=0}^kg_i\Stab(Q_i)g_i^{-1}$ which follows from Lemma~\ref{lem_Hfistab} and the following well-known facts: 
	\begin{enumerate}
		\item[(i)] If $A_1,\cdots,A_m,B_1,\cdots,B_m$ are subgroups of a group $G$ such that $A_j$ is a finite-index subgroup of $B_j$ for every $1\le j\le m$, then $\bigcap_{j=1}^mA_j$ is finite-index in $\bigcap_{j=1}^mB_j$ (for a proof, see \cite[Lemma 2.1]{MR3008315}). 
		\item[(ii)] If $A,B$ are subgroups of a group $G$ such that $A$ is a finite-index subgroup of $B$, then for every $g\in G$, $gAg^{-1}$ is finite-index in $gBg^{-1}$. 
	\end{enumerate}
	
	We now show that (2) and (3) are equivalent. Suppose that $\bigcap_{i=0}^kg_i(Q_i)\neq\emptyset$. Let $\gamma\in\bigcap_{i=0}^kg_i(Q_i)$. Then $T_\gamma\in g_iG_ig_i^{-1}$ for every $0\le i\le k$. Then $\bigcap_{i=0}^kg_iG_ig_i^{-1}$ is infinite since it contains the subgroup $\langle T_\gamma\rangle$. Conversely, suppose that $\bigcap_{i=0}^kg_iG_ig_i^{-1}$ is infinite. In particular, there is a non-trivial $g\in\bigcap_{i=0}^kg_iG_ig_i^{-1}$ such that, for each $0\le i\le k$, there are some $a^i_1,\cdots,a^i_{t(i)}\in g_i(Q_i)$ and some $x^i_1,\cdots,x^i_{t(i)}\in\mathbb{Z}-\{0\}$ such that \[g=T_{a^0_1}^{x^0_1}\cdots T_{a^0_{t(0)}}^{x^0_{t(0)}}=\cdots=T_{a^k_1}^{x^k_1}\cdots T_{a^k_{t(k)}}^{x^k_{t(k)}}.\] By Lemma~\ref{lem_3.17FM12}, $t(0)=\cdots =t(k)$ and $\{T_{a^0_i}^{x^0_i}\mid 1\le i\le t(0)\}=\cdots=\{T_{a^k_i}^{x^k_i}\mid 1\le i\le t(k)\}\neq\emptyset$. Then Lemma~\ref{lem_Dtp}(2) implies that $\{a_i^0\mid 1\le i\le t(0)\}=\cdots=\{a_i^k\mid 1\le i\le t(k)\}\subseteq g_i(Q_i)$ for every $0\le i\le k$, hence $\bigcap_{i=0}^kg_i(Q_i)$ is non-empty.
	
	For the last assertion, if $g\in\bigcap_{i=0}^kg_iG_ig_i^{-1}$, the argument above shows that $g\in\langle T_a\mid a\in\bigcap_{i=0}^kg_i(Q_i)\rangle$. Conversely, if $g\in\langle T_a\mid a\in\bigcap_{i=0}^kg_i(Q_i)\rangle$, then $g\in g_iG_ig_i^{-1}$ for every $0\le i\le k$, so $g\in\bigcap_{i=0}^kg_iG_ig_i^{-1}$.
\end{proof}

We are now ready to prove Theorem~\ref{thm_combequi}(1), that is, $\mathcal{N}(\mathcal{C}(S))$ is isomorphic to $\mathcal{K}(\Mod(S),\mathcal{T})$. Recall that $\mathcal{N}(\mathcal{C}(S))$ is the nerve of the collection of maximal simplices of $\mathcal{C}(S)$, and $\mathcal{T}$ is the collection of $\Mod(S)$-stabilisers of representatives $P_1,\cdots,P_n$ of $\Mod(S)$-orbits of the action on the maximal simplices of $\mathcal{C}(S)$. 

\begin{proof}[Proof of Theorem~\ref{thm_combequi}(1)]
	Define $\varphi\colon\mathcal{K}(\Mod(S),\mathcal{T})\to\mathcal{N}(\mathcal{C}(S))$ by $\varphi(g\Stab(P_t))=g(P_t)$ for each $1\le t\le n$. For $g,h\in\Mod(S)$, if $g\Stab(P_a)=h\Stab(P_b)$ then $a=b$, and $g\Stab(P_t)=h\Stab(P_t)$ if and only if $g(P_t)=h(P_t)$, so $\varphi$ is well-defined and injective. For every $P\in\mathcal{P}$, there are $x\in\Mod(S)$ and $1\le j\le n$ with $x(P_j)=P$, so $\varphi(x\Stab(P_j))=x(P_j)=P$. This shows that $\varphi$ is surjective. 
	
	We now show that $\varphi,\varphi^{-1}$ are simplicial maps. Let $g_0\Stab(Q_0),\cdots,g_k\Stab(Q_k)$ be vertices of $\mathcal{K}(\Mod(S),\mathcal{T})$. By Lemma~\ref{lem_comgen}, $\{g_0\Stab(Q_0),\cdots,g_k\Stab(Q_k)\}$ is a simplex in $\mathcal{K}(\Mod(S),\mathcal{T})$ if and only if $\{g_0(Q_0),\cdots,g_k(Q_k)\}$ is a simplex in $\mathcal{N}(\mathcal{C}(S))$. It follows that $\varphi$ and $\varphi^{-1}$ are simplicial. This proves that $\varphi$ is a simplicial isomorphism. 
\end{proof}

\begin{remark}
	Note that $\mathcal{K}(\Mod(S),\mathcal{H})$ is not isomorphic to $\mathcal{N}(\mathcal{C}(S))$. For each $1\le i\le n$, $H_i$ is not isomorphic to $\Stab(P_i)$, so there are $g,h\in H_i$ with $gH_i\neq hH_i$ and $g(P_i)=h(P_i)$. 
\end{remark}

\section{Automorphism group of the nerve}
\label{sec_autx}

In this section, we prove two general results for simplicial complexes, Theorems~\ref{thm_nervenerve} and \ref{thm_assumpautxiso}. Then we use these results to deduce  
Theorem~\ref{thm_combequi}(2) and Theorem~\ref{thm_metacicmodsT} from the introduction. 

Let $X$ be a finite-dimensional simplicial complex. Recall that $\mathcal{N}(X)$ denotes the nerve of the collection of maximal simplices of $X$. Analogously, 
$\mathcal{N}(\mathcal{N}(X))$ is the nerve of the collection of maximal simplices of $\mathcal{N}(X)$.  For every $x\in X^{[0]}$, define $\Sigma_x$ as the collection of maximal simplices of $X$ that contain the vertex $x$. 

\begin{theorem}
	\label{thm_nervenerve}
	Suppose that for every $x,y\in X^{[0]}$, $\Sigma_x\subseteq\Sigma_y$ implies that $x=y$. Then there is a simplicial isomorphism between $X$ and $\mathcal{N}(\mathcal{N}(X))$.
\end{theorem}

\begin{theorem}
	\label{thm_assumpautxiso}
	Suppose the following two conditions hold for $X$. 
	\begin{enumerate}
		\item For every $x,y\in X^{[0]}$, $\Sigma_x\subseteq\Sigma_y$ implies that $x=y$.
		\item If $K$ is a simplex of $X$ and $x$ is a vertex of $X$ such that $x\notin K$, then there is a maximal simplex $L$ of $X$ such that $K\subseteq L$ and $x\notin L$.
	\end{enumerate}
	Then $\Aut(X)$ is isomorphic to $\Aut(\mathcal{N}(X))$. 
\end{theorem}

\begin{lemma}
	\label{lem_simpnonemp}
	Let $x_0,\cdots,x_t\in X^{[0]}$. Then $\{x_0,\cdots,x_t\}$ is a simplex in $X$ if and only if $\bigcap_{i=0}^{t}\Sigma_{x_i}\neq\emptyset$. 
\end{lemma}

\begin{proof}
	Assume that $\{x_0,\cdots,x_t\}$ is a simplex in $X$. Then there is a maximal simplex $K$ in $X$ containing $\{x_0,\cdots,x_t\}$ as a face, and hence $K\in\bigcap_{i=0}^{t}\Sigma_{x_i}\neq\emptyset$.  Conversely, if $\bigcap_{i=0}^{t}\Sigma_{x_i}\neq\emptyset$ then there is  $L\in\bigcap_{i=0}^{t}\Sigma_{x_i}$ and such simplex $L$ of $X$ contains $\{x_0,\cdots,x_t\}$ as a face. 
\end{proof}

\begin{lemma}
	\label{lem_maxsimexi}
	Let $\Sigma$ be a maximal simplex of $\mathcal{N}(X)$. Then there is some $x\in X^{[0]}$ with $\Sigma=\Sigma_x$. Assume further that, for every $a,b\in X^{[0]}$, $\Sigma_a\subseteq\Sigma_b$ implies that $a=b$. Then for every $y\in X^{[0]}$, $\Sigma_y$ is a maximal simplex of $\mathcal{N}(X)$. 
\end{lemma}

\begin{proof}
	If $\bigcap\Sigma\neq\emptyset$, then there is $x\in X^{[0]}$ such that $x\in L$ for every $L\in\Sigma$, so $\Sigma\subseteq\Sigma_x$, and then $\Sigma=\Sigma_x$ by maximality of $\Sigma$. Therefore, it suffices to show that $\bigcap\Sigma\neq\emptyset$. Assume for a contradiction that $\bigcap\Sigma=\emptyset$. Let $K\in\Sigma$ and suppose $K=\{x_0,\cdots,x_d\}$. Then for every $0\le i\le d$, there is $L_i\in\Sigma$ such that $x_i\notin L_i$. This means $\{K,L_0,\cdots,L_d\}\subseteq\Sigma$ is a collection of maximal simplices in $X$ such that $(\bigcap_{i=0}^dL_i)\bigcap K=\emptyset$, and hence $\Sigma$ is not a maximal simplex in $\mathcal{N}(X)$, a contradiction. 
	
	Suppose that, for every $a,b\in X^{[0]}$, $\Sigma_a\subseteq\Sigma_b$ implies that $a=b$. Let $y\in X^{[0]}$. Since $\Sigma_y$ is a simplex of $\mathcal{N}(X)$, there is a maximal simplex $\Phi$ with $\Sigma_y\subseteq\Phi$, so $\Phi=\Sigma_z$ for some $z\in X^{[0]}$. Hence $y=z$. 
\end{proof}

\begin{proof}[Proof of Theorem~\ref{thm_nervenerve}]
	Define $\beta\colon X\to\mathcal{N}(\mathcal{N}(X))$ by $\beta(x)=\Sigma_x$. By Lemma~\ref{lem_maxsimexi}, the function $\beta$ is well-defined and surjective. By assumption, $\Sigma_x=\Sigma_y$ if and only if $x=y$, therefore the map $\beta$ is injective and hence a bijection. It remains to show that $\beta$ and $\beta^{-1}$ are simplicial. Let $\tau=\{x_0,\cdots,x_r\}$ be a set of vertices of $X$. By Lemma~\ref{lem_simpnonemp}, $\tau$ is a simplex of $X$ if and only if $\bigcap_{i=0}^r\Sigma_{x_i}\neq\emptyset$. On the other hand, $\{\Sigma_{x_0},\cdots,\Sigma_{x_r}\}$ is a simplex of $\mathcal{N}(\mathcal{N}(X))$ if and only if $\bigcap_{i=0}^r\Sigma_{x_i}\neq\emptyset$. Therefore $\beta$ and $\beta^{-1}$ are simplicial, in particular, $\beta$ is a simplicial isomorphism. 
\end{proof}

Now we prove Theorem~\ref{thm_assumpautxiso}. Define  
\[\Omega\colon\Aut(X)\to\Aut(\mathcal{N}(X)), \qquad\Omega(\varphi)=\varphi_*,\]
where for each $\varphi\in\Aut(X)$, the map $\varphi_*\colon\mathcal{N}(X)\to\mathcal{N}(X)$ is given by $\varphi_*(K)=\varphi(K)$. The function $\Omega$ is well-defined by the following lemma.

\begin{lemma}
	If $\varphi\in\Aut(X)$, then $\varphi_*\colon\mathcal{N}(X)\to\mathcal{N}(X)$ is a simplicial automorphism.
\end{lemma}

\begin{proof}
	Note that $\varphi_*$ is a bijection. It remains to show that $\varphi_*$ and $(\varphi_*)^{-1}$ are simplicial maps. Let $L=\{K_0,\cdots,K_t\}$ be a simplex in $\mathcal{N}(X)$. Then $\bigcap_{i=0}^tK_i\neq\emptyset$. Since $\varphi$ is bijective, we have $\bigcap_{i=0}^t\varphi(K_i)=\varphi(\bigcap_{i=0}^tK_i)\neq\emptyset$, so $\varphi_*(L)$ is a simplex in $\mathcal{N}(X)$. This shows that $\varphi_*$ is a simplicial map. Since $(\varphi_*)^{-1}=(\varphi^{-1})_*$, we have that $(\varphi_*)^{-1}$ is simplicial.
\end{proof}

\begin{remark}
	\label{lem_phisigxsigphix}
	If $\varphi\in\Aut(X)$ and $x$ is a vertex of $X$, then $\varphi_*(\Sigma_x)=\Sigma_{\varphi(x)}$.
\end{remark}

% \begin{proof}
	%     This is a direct consequence that $\varphi$ maps maximal simplices of $X$ to maximal simplices of $X$. 
% \end{proof}

For a simplex $\tau$ of $X$, denote $\Sigma_\tau$ as the collection of all maximal simplices of $X$ that contain $\tau$ as a face.

\begin{lemma}
	\label{lem_inttpts}
	If the second condition of Theorem~\ref{thm_assumpautxiso} holds, $\bigcap\Sigma_\tau=\tau$ for every simplex $\tau$ in $X$. In particular, $\bigcap\Sigma_x=\{x\}$.
\end{lemma}

\begin{proof}
	Observe that $\tau\subseteq\bigcap\Sigma_\tau$. To show that $\bigcap\Sigma_\tau\subseteq\tau$, assume for a contradiction that there is $a\in\bigcap\Sigma_\tau$ such that $a\notin\tau$. By assumption on $X$, there is a maximal simplex $L$ in $X$ with $\tau\subseteq L$ and $a\notin L$. Since $L\in\Sigma_\tau$, we have that $a\notin\bigcap\Sigma_\tau$, a contradiction. 
\end{proof}

\begin{prop}
	\label{prop_omeinj}
	Suppose that for every $x,y\in X^{[0]}$, $\Sigma_x\subseteq\Sigma_y$ implies that $x=y$. Then $\Omega$ is injective.
\end{prop}

\begin{proof}
	Let $\varphi,\psi\in\Aut(X)$. Suppose that $\varphi_*=\psi_*$. Let $x\in X^{[0]}$. By Lemma~\ref{lem_inttpts}, we have $\bigcap\Sigma_x=\{x\}$. Then by Remark \ref{lem_phisigxsigphix}, we have \[\{\varphi(x)\}=\bigcap\Sigma_{\varphi(x)}=\bigcap\varphi_*(\Sigma_x)=\varphi_*\left(\bigcap\Sigma_x\right)=\psi_*\left(\bigcap\Sigma_x\right)=\bigcap\psi_*(\Sigma_x)=\bigcap\Sigma_{\psi(x)}=\{\psi(x)\}.\] Since $X$ is abstract simplicial complex, this shows that $\varphi=\psi$.
\end{proof}

\begin{prop}
	\label{prop_omesurj}
	Suppose that $X$ satisfies two conditions in Theorem~\ref{thm_assumpautxiso}. Then $\Omega$ is surjective.
\end{prop}

\begin{proof}
	Let $g\in\Aut(\mathcal{N}(X))$. Lemma~\ref{lem_maxsimexi} shows that every maximal simplex in $\mathcal{N}(X)$ is of the form $\Sigma_a$ for some $a\in X^{[0]}$, and note that $g$ maps maximal simplices to maximal simplices bijectively. Then for each $x\in X^{[0]}$, there is $y\in X^{[0]}$ such that $g(\Sigma_x)=\Sigma_{y}$. This vertex $y$ is unique since $\bigcap g(\Sigma_x)=\bigcap\Sigma_{y}=\{y\}$ by Lemma~\ref{lem_inttpts}. Define $\varphi\colon X\to X$ by $\varphi(x)=y$. 
	
	We first show that $\varphi$ is bijective. Let $x,y\in X^{[0]}$ and suppose that $\varphi(x)=\varphi(y)$. Note that $g(\Sigma_x)=\Sigma_{\varphi(x)}$ and $\Sigma_{\varphi(y)}=g(\Sigma_y)$, so $\Sigma_x=\Sigma_y$ since $g$ is bijective, and then $x=y$ by our assumption. This proves that $\varphi$ is injective. For each $a\in X^{[0]}$, there is $b\in X^{[0]}$ such that $g^{-1}(\Sigma_a)=\Sigma_b$, then $g(\Sigma_b)=\Sigma_a$, so $\varphi(b)=a$. This proves that $\varphi$ is surjective. 
	
	We now verify that $\varphi$ and $\varphi^{-1}$ are simplicial maps. Let $\{x_0,\cdots,x_k\}$ be a simplex in $X$. Then by Lemma~\ref{lem_simpnonemp}, we have $\bigcap_{i=0}^k\Sigma_{x_i}\neq\emptyset$. Since $g$ is bijective, we have \[\bigcap_{i=0}^k\Sigma_{\varphi(x_i)}=\bigcap_{i=0}^kg(\Sigma_{x_i})=g\left(\bigcap_{i=0}^k\Sigma_{x_i}\right)\neq\emptyset,\] so $\{\varphi(x_0),\cdots,\varphi(x_k)\}$ is a simplex in $X$ by Lemma~\ref{lem_simpnonemp}. Hence $\varphi$ is a simplicial map. Similar argument on $\varphi^{-1}$ shows that it is a simplicial map. Therefore, we have $\varphi\in\Aut(X)$. 
	
	It remains to show that $\varphi_*=g$. Let $\sigma=\{v_0,\cdots,v_t\}$ be a maximal simplex of $X$. Then $\varphi_*(\sigma)=\varphi(\sigma)=\{\varphi(v_0),\cdots,\varphi(v_t)\}$, which is a maximal simplex in $X$. Then by Lemma~\ref{lem_inttpts}, \[g(\sigma)=g\left(\bigcap_{i=0}^t\Sigma_{v_i}\right)=\bigcap_{i=0}^tg(\Sigma_{v_i})=\bigcap_{i=0}^t\Sigma_{\varphi(v_i)}=\{\varphi(v_0),\cdots,\varphi(v_t)\}.\] This shows that $\varphi_*=g$.
\end{proof}

\begin{proof}[Proof of Theorem~\ref{thm_assumpautxiso}]
	This follows from Propositions~\ref{prop_omeinj} and \ref{prop_omesurj}. 
\end{proof}

To prove Theorem~\ref{thm_combequi}(2) and Theorem~\ref{thm_metacicmodsT}, it suffices to show the following proposition.

\begin{prop}
	\label{prop_ccsp}
	The curve complex $\mathcal{C}(S)$ satisfies the following two properties. 
	\begin{enumerate}
		\item For every pair of vertices $x,y$ in $\mathcal{C}(S)$, $\Sigma_x\subseteq\Sigma_y$ implies that $x=y$.
		\item If $K$ is a simplex of $\mathcal{C}(S)$ and $x$ is a vertex of $\mathcal{C}(S)$ such that $x\notin K$, then there is a maximal simplex $L$ of $\mathcal{C}(S)$ such that $K\subseteq L$ and $x\notin L$.
	\end{enumerate}
\end{prop} 

To prove this proposition, we need the following lemma. 

\begin{lemma}[{\cite[Proposition 2.3]{MR2952767}}]
	\label{lem_elemove}
	Let $C=\{a_1,\cdots,a_k\}$ be a simplex in $\mathcal{C}(S)$. Then for every $1\le t\le k$, there is a vertex $b$ of $\mathcal{C}(S)$ such that $i(b,a_t)>0$ and $i(b,a)=0$ for every $a\in C-\{a_t\}$. 
\end{lemma}

\begin{proof}[Proof of Proposition~\ref{prop_ccsp}]
	To show the first property, let $a,b$ be vertices of $\mathcal{C}(S)$ with $\Sigma_a\subseteq\Sigma_b$. Assume for a contradiction that $a\neq b$. Let $A\in\Sigma_a$. Then $A\in\Sigma_b$, so $a,b\in A$. We can write $A=\{c_1=a,c_2,\cdots,c_{3g-3}=b\}$. By Lemma~\ref{lem_elemove}, there is a vertex $c$ of $\mathcal{C}(S)$ such that $i(c,c_i)=0$ for every $1\le i\le3g-4$ and $i(c,c_{3g-3})>0$, so $B=\{c_1=a,c_2,\cdots,c_{3g-4},c\}$ is a maximal simplex in $\mathcal{C}(S)$ with $B\in\Sigma_a$ and $B\notin\Sigma_b$, contradicting $\Sigma_a\subseteq\Sigma_b$. Hence $a=b$.
	
	To show the second property, let $K$ be a simplex of $\mathcal{C}(S)$. Let $x$ be a vertex of $\mathcal{C}(S)$ with $x\notin K$. Write $K=\{x_0,\cdots,x_n\}$. If $i(x,x_j)>0$ for some $0\le j\le n$, then we can take any maximal simplex $L$ in $\mathcal{C}(S)$ with $K\subseteq L$ and note that $x\notin L$. If $i(x,x_i)=0$ for every $0\le i\le n$, then $D=\{x_0,\cdots,x_n,x\}$ is a simplex in $\mathcal{C}(S)$. Let $E$ be a maximal simplex in $\mathcal{C}(S)$ with $D\subseteq E$. Then we apply Lemma~\ref{lem_elemove} on $E$ to obtain a maximal simplex $L$ such that $K\subseteq L$ and $x\notin L$. 
\end{proof}

We are now able to prove Theorem~\ref{thm_combequi}(2) and Theorem~\ref{thm_metacicmodsT}.

\begin{proof}[Proof of Theorem~\ref{thm_combequi}(2)]
	The curve complex $\mathcal{C}(S)$ is finite-dimensional, and by Proposition~\ref{prop_ccsp} it satisfies the condition of Theorem~\ref{thm_nervenerve},  so $\mathcal{C}(S)$ is isomorphic to $\mathcal{N}(\mathcal{N}(\mathcal{C}(S)))$.  Since $\mathcal{N}(\mathcal{C}(S))$ is isomorphic to $\mathcal{K}(\Mod(S),\mathcal{T})$ by Theorem~\ref{thm_combequi}(1), $\mathcal{C}(S)$ is isomorphic to $\mathcal{N}(\mathcal{K}(\Mod(S),\mathcal{T}))$.
\end{proof}

\begin{proof}[Proof of Theorem~\ref{thm_metacicmodsT}]
	By Proposition~\ref{prop_ccsp}, the finite-dimensional simplicial complex $\mathcal{C}(S)$ satisfies both conditions of Theorem~\ref{thm_assumpautxiso}, so $\Aut(\mathcal{C}(S))$ is isomorphic to $\Aut(\mathcal{N}(\mathcal{C}(S)))$. 
\end{proof}

\section{Groups quasi-isometric to the mapping class groups}
\label{sec_gqimcgs}

We prove Corollary~\ref{thm_gqimcg} in this section. All groups considered in this section are finitely generated, and each of them is endowed with a word metric induced by a finite generating set. 

\begin{defn}
	A pair of subgroups $H_1$ and $H_2$ of a group $G$ are \emph{commensurable} if $H_1\cap H_2$ is finite-index in both $H_1$ and $H_2$. Given a subgroup $K$ of $G$, the \emph{commensurator} of $K$ is defined as $\Comm_G(K)=\{g\in G\mid K\text{ and }gKg^{-1}\text{ are commensurable}\}$.
\end{defn} 

The following proposition is well-known.

\begin{prop}
	\label{prop_ficom}
	Let $A,B$ be subgroups of a group $G$. Let $g\in G$. Then:
	\begin{enumerate}
		\item $\Comm_G(gAg^{-1})=g\Comm_G(A)g^{-1}$.
		\item If $A$ is a finite-index subgroup of $B$, then $\Comm_G(A)=\Comm_G(B)$.
	\end{enumerate}
\end{prop}

\begin{defn}
	A group pair $(G,\mathcal{A})$ is \emph{reducible} if every $A\in\mathcal{A}$ is finite-index in $\Comm_G(A)$, and it is \emph{reduced} if $A=\Comm_G(A)$ for every $A\in\mathcal{A}$ and subgroups in $\mathcal{A}$ are pairwise non-conjugate. 
\end{defn}

Every reduced group pair is reducible. For the remainder of this article, we write $\Comm_{\Mod(S)}(K)$ as $\Comm(K)$ for simplicity. Recall that $P_1,\cdots,P_n$ are representatives of orbits of $\Mod(S)$-action on the set of maximal simplices of $\mathcal{C}(S)$, and $H_i=\langle T_a\mid a\in P_i\rangle$ for each $1\le i\le n$. Each $H_i$ is a free abelian group of rank $3g-3$, where $g\ge2$ is the genus of the surface $S$. 

\begin{prop}
	\label{prop_Treduced}
	For each $1\le i\le n$, we have $\Comm(H_i)=\Stab(P_i)$. Consequently, $(\Mod(S),\mathcal{T})$ is reduced, and so reducible.
\end{prop}

\begin{proof}
	Let $g\in\Stab(P_i)$. Then $g(P_i)=P_i$, so $gH_ig^{-1}=H_i$ by Lemma~\ref{lem_Dtp}(3). It follows that $g\in\Comm(H_i)$. Conversely, let $g\in\Comm(H_i)$. Then $gH_ig^{-1}$ and $H_i$ are commensurable and are both free abelian groups of rank $3g-3$ by Lemma~\ref{lem_Dtp}(3), so $gH_ig^{-1}\cap H_i$ is a free abelian group of rank $3g-3$. Then by Lemma~\ref{lem_comgen}, $gH_ig^{-1}=H_i$, so $g(P_i)=P_i$, i.e. $g\in\Stab(P_i)$. This shows that $\Comm(H_i)=\Stab(P_i)$. By Lemma~\ref{lem_Hfistab}, $H_i$ is a finite-index subgroup of $\Stab(P_i)$, so $\Comm(H_i)=\Comm(\Stab(P_i))=\Stab(P_i)$ by Proposition~\ref{prop_ficom}.
	
	Suppose that $g\Stab(P_i)g^{-1}=h\Stab(P_j)h^{-1}$ for some $g,h\in\Mod(S)$ and some $1\le i,j\le n$. Then for every $a\in P_i$, $T_{h^{-1}(g(a))}\in\Stab(P_j)$, which implies that $h^{-1}(g(P_i))=P_j$ by Lemma~\ref{lem_Dtp}(5), so $P_i=P_j$. Therefore, subgroups in $\mathcal{T}$ are pairwise non-conjugate. 
\end{proof}

We also need the definition of virtually isomorphic group pairs. 

\begin{defn}[{\cite[Definition 2.8]{MR5034311}}]
	\label{def_vi}
	Let $(G,\mathcal{A})$ and $(K,\mathcal{B})$ be two group pairs. We say that $(G,\mathcal{A})$ and $(K,\mathcal{B})$ are \emph{virtually isomorphic} if one can be transformed to another by a finite sequence of the following operations and their inverses:
	\begin{enumerate}
		\item Substitute $\mathcal{A}$ with a finite collection $\mathcal{B}$ such that each $B\in\mathcal{B}$ is commensurable to a conjugate of $A\in\mathcal{A}$, and vice versa.
		\item Substitute $(G,\mathcal{A})$ with $(G/N,\{AN/N\mid A\in\mathcal{A}\})$, where $N$ is a finite normal subgroup of $G$. 
		\item Substitute $(G,\mathcal{A})$ with $(K,\mathcal{B})$, where $K$ is a finite-index subgroup of $G$, and $\mathcal{B}$ is given by: if $\{g_iA_i\mid i\in I\}$ is a collection of representatives of the orbits of the action of $K$ on $G/\mathcal{A}$ by left multiplication, then $\mathcal{B}=\{g_iA_ig_i^{-1}\cap K\mid i\in I\}$. 
	\end{enumerate}
\end{defn}

A map $f\colon X\to Y$ between metric spaces is a $(\lambda,\mu,C)$-\emph{quasi-isometry} for $\lambda\ge1$, $\mu,C\ge0$ if for every $a,b\in X$,
$\frac{1}{\lambda}d_X(x,y)-\mu\le d_Y(f(x),f(y))\le\lambda d_X(x,y)+\mu$,
and for every $y\in Y$, there is $x\in X$ with $d_Y(f(x),y)\le C$.

A $(\lambda,\mu,C,M)$-\emph{quasi-isometry of group pairs} $q\colon(G,\mathcal{A})\to(K,\mathcal{B})$ is a $(\lambda,\mu,C)$-quasi-isometry $q\colon G\to K$ such that for every $gA\in G/\mathcal{A}$ there is $kB\in K/\mathcal{B}$ such that $\hdist(q(gA),kB)<M$, where $\hdist(\cdot,\cdot)$ denotes the Hausdorff distance in $K$, and the quasi-inverse of $q$ has the analogous property, see \cite[Section 2]{MR5034311}. This gives an equivalence relation in the class of group pairs. 

\begin{prop}[{\cite[Proposition 2.9]{MR5034311}}]
	\label{prop_viqi}
	If two group pairs are virtually isomorphic, then they are quasi-isometric group pairs. 
\end{prop}

The following important result is due, independently, to Behrstock\textendash Kleiner\textendash Minsky\textendash Mosher \cite{MR2928983} and Hamenst\"{a}dt \cite{hamenstaedt2007geometrymappingclassgroups}.

\begin{theorem}[{\cite[Theorem 1.2]{MR2928983}}]
	\label{thm_MCGrigidity}
	Let $K$ be a finitely generated group quasi-isometric to $\Mod^{\pm}(S)$. Then there is a group homomorphism $K\to\Mod^{\pm}(S)$ with finite kernel and finite-index image. 
\end{theorem}

A corollary of this result that we will use is the following. 

\begin{cor}
	\label{lem_HQModST}
	Let $H$ be a finitely generated group quasi-isometric to $\Mod(S)$. Then there is a finite collection $\mathcal{Q}$ of subgroups of $H$ such that $(H,\mathcal{Q})$ is quasi-isometric to $(\Mod(S),\mathcal{T})$.
\end{cor}

\begin{proof}
	We first show that $(\Mod^{\pm}(S),\mathcal{T})$ is virtually isomorphic to $(\Mod(S),\mathcal{T})$ using Definition~\ref{def_vi}(1)(2). Let $\{h_i\Stab(P_i)\mid1\le i\le k\}$ be representitives of $\Mod(S)$-action on $\Mod^{\pm}(S)/\mathcal{T}$. Then $(\Mod^{\pm}(S),\mathcal{T})$ is virtually isomorphic to $(\Mod(S),\mathcal{R})$, where $\mathcal{R}=\{h_i\Stab(P_i)h_i^{-1}\mid1\le i\le k\}$. For each $1\le i\le k$, there are $f\in\Mod(S)$ and $1\le j\le k$ such that $f(P_j)=h_i(P_i)$, so $f\Stab(P_j)f^{-1}=h_i\Stab(P_i)h_i^{-1}$. It follows that $(\Mod(S),\mathcal{R})$ is virtually isomorphic to $(\Mod(S),\mathcal{T})$.
	
	We next prove that $(\Mod^{\pm}(S),\mathcal{T})$ is virtually isomorphic to a group pair $(H,\mathcal{A})$ by applying Definition~\ref{def_vi}(2)(3). By Theorem~\ref{thm_MCGrigidity}, there is a group homomorphism $f\colon H\to\Mod^{\pm}(S)$ such that $\ker(f)$ is finite and $f(H)$ is finite-index in $\Mod^{\pm}(S)$. Let $\{g_i\Stab(P_i)\mid1\le i\le t\}$ be representatives of $f(H)$-action on $\Mod^{\pm}(S)/\mathcal{T}$, and set $\mathcal{B}=\{g_i\Stab(P_i)g_i^{-1}\cap f(H)\mid1\le i\le t\}$. Then $(f(H),\mathcal{B})$ is virtually isomorphic to $(\Mod^{\pm}(S),\mathcal{T})$. Note that $(f(H),\mathcal{B})$ is the same as $(H/\ker(f),\{f^{-1}(B)/\ker(f)\mid B\in\mathcal{B}\})$, and the latter pair is virtually isomorphic to $(H,\mathcal{A})$ where $\mathcal{A}=\{f^{-1}(B)\mid B\in\mathcal{B}\}$.  
	
	Finally, Proposition~\ref{prop_viqi} implies that $(\Mod(S),\mathcal{T})$ is quasi-isometric to $(H,\mathcal{A})$.
\end{proof}

We also need the following three results for the main theorem of this section.

\begin{theorem}[{\cite[Theorem 3.9]{MR5034311}}]
	\label{thm_reduqiinv}
	Let $(G,\mathcal{A})$ and $(K,\mathcal{B})$ be two quasi-isometric group pairs such that $(G,\mathcal{A})$ is reducible. Then $(K,\mathcal{B})$ is also reducible. 
\end{theorem}

\begin{defn}[{\cite[Definition 3.17]{MR5034311}}]
	Let $\mathcal{A}$ be a finite collection of subgroups of a group $G$. A \emph{refinement} of $\mathcal{A}$, denoted by $\mathcal{A}^*$, is a set of representatives of conjugacy classes of the collection $\{\Comm_G(gAg^{-1})\mid A\in\mathcal{A},g\in G\}$.
\end{defn}

\begin{prop}[{\cite[Remark 3.18]{MR5034311}}]
	\label{prop_refinement}
	Let $(G,\mathcal{A})$ be a reducible group pair. Then $(G,\mathcal{A}^*)$ is reduced and $(G,\mathcal{A})$ is virtually isomorphic to $(G,\mathcal{A}^*)$.
\end{prop}

\begin{prop}[{\cite[Proposition 4.9(3)]{MR5034311}}]
	\label{prop_reducedsimiso}
	Let $q\colon(G,\mathcal{A})\to(K,\mathcal{B})$ be a quasi-isometry of group pairs. If both $(G,\mathcal{A})$ and $(K,\mathcal{B})$ are reduced, then $q$ induces a simplicial isomorphism $\dot{q}\colon\mathcal{K}(G,\mathcal{A})\to\mathcal{K}(K,\mathcal{B})$.
\end{prop}

\begin{theorem}
	\label{thm_ModTHQqi}
	Let $H$ be a finitely generated group quasi-isometric to $\Mod(S)$. Then there is a finite collection $\mathcal{R}$ of subgroups of $H$ such that $(\Mod(S),\mathcal{T})$ is quasi-isometric to $(H,\mathcal{R})$. Moreover, the coset intersection complexes $\mathcal{K}(\Mod(S),\mathcal{T})$ and $\mathcal{K}(H,\mathcal{R})$ are isomorphic. 
\end{theorem}

\begin{proof}
	By Corollary~\ref{lem_HQModST}, there is a collection $\mathcal{Q}$ of subgroups of $H$ such that $(H,\mathcal{Q})$ is quasi-isometric to $(\Mod(S),\mathcal{T})$, which is reducible by Proposition~\ref{prop_Treduced}. Then Theorem~\ref{thm_reduqiinv} implies that $(H,\mathcal{Q})$ is reducible, so it is quasi-isometric to $(H,\mathcal{Q}^*)$ by Propositions~\ref{prop_viqi} and \ref{prop_refinement}. Therefore, $(\Mod(S),\mathcal{T})$ and $(H,\mathcal{Q}^*)$ are quasi-isometric and reduced by Propositions~\ref{prop_Treduced} and \ref{prop_refinement}, so $\mathcal{K}(\Mod(S),\mathcal{T})$ and $\mathcal{K}(H,\mathcal{Q}^*)$ are isomorphic by Proposition~\ref{prop_reducedsimiso}. Setting $\mathcal{R}=\mathcal{Q}^*$ completes the proof. 
\end{proof}

\begin{proof}[Proof of Corollary~\ref{thm_gqimcg}]
	Theorem~\ref{thm_ModTHQqi} provides a group pair $(H,\mathcal{R})$ such that $\mathcal{K}(\Mod(S),\mathcal{T})$ is isomorphic to $\mathcal{K}(H,\mathcal{R})$, so they are both isomorphic to $\mathcal{N}(\mathcal{C}(S))$ by Theorem~\ref{thm_combequi}(1). Moreover, $\mathcal{N}(\mathcal{K}(\Mod(S),\mathcal{T}))$ and $\mathcal{N}(\mathcal{K}(H,\mathcal{R}))$ are isomorphic, so they are both isomorphic to $\mathcal{C}(S)$ by Theorem~\ref{thm_combequi}(2).
\end{proof}

\section{Quasi-isometry and homotopy equivalence}
\label{sec_qihe}

We prove Theorem~\ref{thm_qihe} in this section. We first recall the following classical result.

\begin{theorem}[Nerve theorem, {\cite[Lemma 1.1]{MR607041}}]
	\label{thm_NerveThm}
	Let $X$ be a simplicial complex. Let $\mathcal{U}=\{U_i\}_{i\in I}$ be a cover for $X$ such that every finite intersection of elements in $\mathcal{U}$ is either empty or contractible. Then $X$ is homotopy equivalent to the nerve of $\mathcal{U}$. 
\end{theorem}

\begin{cor}
	\label{cor_cshmkt}
	The curve complex $\mathcal{C}(S)$ is homotopy equivalent to $\mathcal{K}(\Mod(S),\mathcal{T})$. 
\end{cor}
\begin{proof}
	Recall that $\mathcal{P}$ denotes the collection of all maximal simplices of $\mathcal{C}(S)$.  The intersection of any two elements in $\mathcal{P}$ is either empty or a simplex, so every finite intersection of elements in $\mathcal{P}$ is either empty or a simplex.  By the nerve theorem, the curve complex $\mathcal{C}(S)$ is homotopy equivalent to $\mathcal{N}(\mathcal{C}(S))$. Then the corollary follows from Theorem~\ref{thm_combequi}(1). 
\end{proof}

\begin{prop}
	\label{prop_csqik}
	The curve complex $\mathcal{C}(S)$ is quasi-isometric to $\mathcal{K}(\Mod(S),\mathcal{T})$.
\end{prop}

\begin{proof}
	Denote the barycentric subdivision of a simplicial complex $X$ as $\bsd(X)$. Since $\mathcal{C}(S)$ is connected, by Corollary~\ref{cor_cshmkt}, $\mathcal{K}(\Mod(S),\mathcal{T})$ is also connected. Following Bridson \cite{MR1170372}, it suffices to show that $\mathcal{K}(\Mod(S),\mathcal{T})^{[1]}$ and $\bsd(\mathcal{C}(S))^{[1]}$ are quasi-isometric. We regard $\bsd(\mathcal{C}(S))^{[1]}$ and $\mathcal{K}(\Mod(S),\mathcal{T})^{[1]}$ as geodesic metric spaces with metrics $d_{\mathcal{C}}$ and $d_{\mathcal{T}}$, where each edge is regarded as a segment of length one. 
	
	Recall that $P_1,\cdots,P_n$ are representatives of the $\Mod(S)$-orbits of the action of $\Mod(S)$ on the set $\mathcal{P}$ of maximal simplices of $\mathcal{C}(S)$. Define $\Phi\colon\mathcal{K}(\Mod(S),\mathcal{T})^{[0]}\to\bsd(\mathcal{C}(S))^{[0]}$ by mapping the vertex $g\Stab(P_i)$ to the barycentre of the simplex $g(P_i)$ in $\mathcal{C}(S)$, where $1\le i\le n$. We will show the following three properties, then one can verify that $\Phi$ defines a quasi-isometry $\mathcal{K}(\Mod(S),\mathcal{T})^{[1]}\to\bsd(\mathcal{C}(S))^{[1]}$:
	\begin{enumerate} 
		\item For each vertex $\alpha$ of $\bsd(\mathcal{C}(S))$, there is a vertex $gG$ of $\mathcal{K}(\Mod(S),\mathcal{T})$ such that $d_{\mathcal{C}}(\alpha,\Phi(gG))\le1$. 
		\item If two vertices $g\Stab(P_i)$ and $f\Stab(P_j)$ are adjacent in $\mathcal{K}(\Mod(S),\mathcal{T})$, then we have $d_{\mathcal{C}}(\Phi(g\Stab(P_i)),\Phi(f\Stab(P_j)))\le2$.
		\item For each pair $g\Stab(P_i)$ and $f\Stab(P_j)$ of vertices of $\mathcal{K}(\Mod(S),\mathcal{T})$ such that the distance $d_{\mathcal{C}}(\Phi(g\Stab(P_i)),\Phi(f\Stab(P_j)))$ is at most two, then $d_{\mathcal{T}}(g\Stab(P_i),f\Stab(P_j))\le1$.
	\end{enumerate}
	
	For (1), if $\alpha$ is a barycentre of some maximal simplex $\tau$ of $\mathcal{C}(S)$, then there are $x\in\Mod(S)$ and $Q\in\{P_1,\cdots,P_n\}$ such that $x(Q)=\tau$, so $\Phi(x\Stab(Q))=x(Q)$. This means $\Image(\Phi)$ is the set of all barycentres of maximal simplices in $\mathcal{C}(S)$. If $\alpha\in\bsd(\mathcal{C}(S))^{[0]}-\Image(\Phi)$, then there is a barycentre of some maximal simplex of $\mathcal{C}(S)$ that is adjacent to $\alpha$, so $d_{\mathcal{C}}(\alpha,\Phi(yR))\le1$ for some vertex $yR$ of $\mathcal{K}(\Mod(S),\mathcal{T})$.
	
	To show (2), let $g\Stab(P_i)$ and $f\Stab(P_j)$ be adjacent vertices in $\mathcal{K}(\Mod(S),\mathcal{T})$. Then $d_{\mathcal{C}}(\Phi(g\Stab(P_i)),\Phi(f\Stab(P_j)))\le2$ since, by Lemma~\ref{lem_comgen}, there is a vertex $\gamma\in g(P_i)\cap f(P_j)$, and therefore there is a path of length two in $\bsd(\mathcal{C}(S))$ between $\Phi(g\Stab(P_i))$ and $\Phi(f\Stab(P_j))$.  
	
	For (3), let $g\Stab(P_i)$ and $f\Stab(P_j)$ be vertices of $\mathcal{K}(\Mod(S),\mathcal{T})$ such that the distance between $\Phi(g\Stab(P_i))$ and $\Phi(f\Stab(P_j))$ in $\bsd(\mathcal{C}(S))$ is at most two. Since $\Phi(g\Stab(P_i))$ and $\Phi(f\Stab(P_j))$ are barycentres of maximal simplices in $\mathcal{C}(S)$, their distance is either zero or two. If the distance is zero, then $f\Stab(P_i)=g\Stab(P_j)$ by injectivity of $\Phi$. If the distance is two, then there is a vertex $\beta\in g(P_i)\cap f(P_j)$, so Lemma~\ref{lem_comgen} implies that $\{g\Stab(P_i),f\Stab(P_j)\}$ is a simplex in $\mathcal{K}(\Mod(S),\mathcal{T})$. This concludes the proof.
\end{proof}

\begin{prop}
	\label{csqih}
	The curve complex $\mathcal{C}(S)$ is quasi-isometric and homotopy equivalent to the coset intersection complex $\mathcal{K}(\Mod(S),\mathcal{H})$.
\end{prop}

\begin{proof}
	Corollary~\ref{cor_cshmkt} and Proposition~\ref{prop_csqik} shows that $\mathcal{C}(S)$ is quasi-isometric and homotopy equivalent to $(\Mod(S),\mathcal{T})$. By Proposition~\ref{prop_Treduced}, $\Comm(H_i)=\Stab(P_i)$ for each $1\le i\le n$, so $(\Mod(S),\mathcal{H})$ and $(\Mod(S),\mathcal{T})$ are virtually isomorphic. It follows from Proposition~\ref{prop_viqi} that these two group pairs are quasi-isometric, and hence $\mathcal{K}(\Mod(S),\mathcal{H})$ and $\mathcal{K}(\Mod(S),\mathcal{T})$ are quasi-isometric and homotopy equivalent by Theorem~\ref{thm_qicicqihe} below.
\end{proof}

\begin{theorem}[{\cite[Theorem 1.4]{MR5034311}}]
	\label{thm_qicicqihe}
	Let $q\colon(G,\mathcal{A})\to(H,\mathcal{B})$ be a quasi-isometry of group pairs. Then $q$ induces a map $\dot{q}\colon\mathcal{K}(G,\mathcal{A})\to\mathcal{K}(H,\mathcal{B})$ that is a quasi-isometry and homotopy equivalent simplicial map. 
\end{theorem}

\begin{proof}[Proof of Theorem~\ref{thm_qihe}]
	This follows from Corollary~\ref{cor_cshmkt}, Propositions~\ref{prop_csqik} and \ref{csqih}.
\end{proof}

\section{Curve complex as a subcomplex of a coset intersection complex}
\label{sec_ccsubcic}

We have shown that the nerve of the curve complex is isomorphic to a coset intersection complex, so it is natural to consider the following question. 

\begin{qtn}
	Is the curve complex $\mathcal{C}(S)$ isomorphic, as a simplicial complex, to a coset intersection complex of some group pair about $\Mod(S)$?
\end{qtn}

Whilst we are unable to answer the question above, we construct a coset intersection complex that contains the curve complex as a subcomplex, see Theorem~\ref{thm_ccincic} below. 

\begin{lemma}
	\label{lem_commtastaba}
	Let $a\in\mathcal{C}^{[0]}(S)$. Then $\Comm(\langle T_a\rangle)=\Stab(a)$.
\end{lemma}
\begin{proof}
	Let $g\in\Comm(\langle T_a\rangle)$. Then $g\langle T_a\rangle g^{-1}\cap\langle T_a\rangle$ is finite-index in both $g\langle T_a\rangle g^{-1}=\langle T_{g(a)}\rangle$ and $\langle T_a\rangle$, so Lemma~\ref{lem_comgen} implies that $g(a)=a$, i.e. $g\in\Stab(a)$. Conversely, let $g\in\Stab(a)$. Then $g(a)=a$, so $g\langle T_a\rangle g^{-1}=\langle T_{g(a)}\rangle=\langle T_a\rangle$, which implies that $g\in\Comm(\langle T_a\rangle)$. 
\end{proof}

The action of $\Mod(S)$ on the vertex set $\mathcal{C}^{[0]}(S)$ of curve complex has finitely many orbits, say $m$. Let $b_1,\cdots,b_m$ be  representatives of these orbits. 

\begin{theorem}
	\label{thm_ccincic}
	Let $\mathcal{B}=\{\Comm(\langle T_{b_i}\rangle)\mid 1\le i\le m\}$. Then there is a $\Mod(S)$-equivariant simplicial embedding $\mathcal{C}(S)\to\mathcal{K}(\Mod(S),\mathcal{B})$ such that $\mathcal{C}^{[0]}(S)$ is isomorphic to $\Mod(S)/\mathcal{B}$. 
\end{theorem}

\begin{proof}
	Define $f\colon\mathcal{C}^{[0]}(S)\to\Mod(S)/\mathcal{B}$ by $f(g(b_i))=g\Comm(\langle T_{b_i}\rangle)$. This map is $\Mod(S)$-equivariant and surjective. To show $f$ is injective, suppose that $a\Comm(\langle T_x\rangle)=b\Comm(\langle T_y\rangle)$ where $a,b\in\Mod(S)$ and $x,y\in\{b_1,\cdots,b_m\}$. Then $b^{-1}a\Comm(\langle T_x\rangle)=\Comm(\langle T_y\rangle)$. Since $\Comm(\langle T_y\rangle)$ is a subgroup of $\Mod(S)$, $b^{-1}a\in\Comm(\langle T_x\rangle)$, so $\Comm(\langle T_x\rangle)=\Comm(\langle T_y\rangle)$. Then Lemma~\ref{lem_commtastaba} implies that $\Stab(x)=\Stab(y)$. Note that $T_x\in\Stab(x)$, so $T_x(y)=y$, which implies that $i(x,y)=0$ by Lemma~\ref{lem_Dtp}(4). If $x\neq y$, then there is $z\in\mathcal{C}^{[0]}(S)$ such that $i(x,z)=0$ and $i(y,z)>0$, see Lemma~\ref{lem_elemove}, and hence $T_z\in\Stab(x)$ and $T_z\notin\Stab(y)$, a contradiction. This means $x=y$, so $b^{-1}a(x)=x=y$, and it follows that $a(x)=b(y)$. This shows that $f$ is injective. 
	
	It remains to prove that $f$ defines a simplicial map. Let $\{g_0(a_{0}),\cdots,g_k(a_{k})\}$ be a $k$-simplex of $\mathcal{C}(S)$ where each $a_i\in \{b_1,\ldots, b_m\}$ and  $g_i\in \Mod(S)$. For every $0\le i,j\le k$, 
	if $i\neq j$ then  \[T_{g_i(a_i)}\in\Stab(g_j(a_j))=\Comm(\langle T_{g_j(a_j)}\rangle)=\Comm(g_j\langle T_{a_j}\rangle g_j^{-1})=g_j\Comm(\langle T_{a_j}\rangle)g_j^{-1}\] by Lemma~\ref{lem_Dtp}(3)(4) and Proposition~\ref{prop_ficom}(1), so $\langle T_{g_i(a_i)}\mid0\le i\le k\rangle\subseteq\bigcap_{i=0}^kg_i\Comm(\langle T_{a_i}\rangle)g_i^{-1}$. Hence $\{g_i\Comm(\langle T_{a_i}\rangle)\mid0\le i\le k\}$ is a simplex in $\mathcal{K}(\Mod(S),\mathcal{B})$.
\end{proof}

\begin{remark}
	The map $f$ in the proof of Theorem~\ref{thm_ccincic} does not define a simplicial isomorphism. The reason is as follows. Let $g(x),h(y)\in\mathcal{C}^{[0]}(S)$ such that $i(g(x),h(y))=1$, where $g,h\in\Mod(S)$ and $x,y\in\{b_1,\cdots,b_m\}$. Since $S$ is a surface of genus at least two, there is $z\in\mathcal{C}^{[0]}(S)$ such that $i(g(x),z)=i(h(y),z)=0$ (see, for instance, the proof of \cite[Theorem 4.3]{MR2850125}), so $\langle T_z\rangle\in g\Comm(\langle T_x\rangle)g^{-1}\bigcap h\Comm(\langle T_y\rangle)h^{-1}$ by a similar argument above. This means $\{g\Comm(\langle T_x\rangle),h\Comm(\langle T_y\rangle)\}$ is a $1$-simplex in $\mathcal{K}(\Mod(S),\mathcal{B})$, but $\{g(x),h(y)\}$ is not a simplex in $\mathcal{C}(S)$. 
\end{remark}
 
\bibliographystyle{alphaurl} 
\bibliography{xbib}

\end{document}